\documentclass[
final
]{dmtcs-episciences}

\usepackage[utf8]{inputenc}
\usepackage{subfigure}
\usepackage{amssymb,amsmath,amsthm}
\usepackage{tikz}
\usetikzlibrary{patterns,snakes}
\usepackage{tabu}

\numberwithin{equation}{section}
\newtheorem{theorem}{Theorem}[section]
\newtheorem{conjecture}{Conjecture}[section]
\newtheorem{lemma}{Lemma}[section]
\newtheorem{corollary}{Corollary}[section]

\theoremstyle{definition}
\newtheorem{remark}{Remark}[section]

\newtheorem*{definition}{Definition}

\newtheorem{example}[theorem]{Example}

\newcommand{\cupdot}{\mathbin{\mathaccent\cdot\cup}}
\newcommand{\makeset}[2]{ \{#1\,:\, #2\} }
\DeclareMathOperator{\slide}{slide}
\DeclareMathOperator{\comp}{Comp}
\DeclareMathOperator{\rgf}{RGF}

\usepackage[round]{natbib}

\author{Jonathan Bloom\affiliationmark{1}
  \and Dan Saracino\affiliationmark{2}}
\title[Pattern avoidance for set partitions \`a la Klazar]{Pattern avoidance for set partitions \`a la Klazar}
\affiliation{
  Lafayette College, Easton, PA, USA\\
  Colgate University, Hamilton, NY, USA}
\keywords{set partitions, Klazar, Wilf-equivalence, restricted growth functions}

\received{2015-11-3}
\revised{2015-3-29}
\accepted{2016-7-26}
\publicationdetails{18}{2016}{2}{9}{1327}

\begin{document}
\publicationdetails{18}{2016}{2}{9}{1327}
\maketitle
\begin{abstract}
 In 2000 Klazar introduced a new notion of pattern avoidance in the context of set partitions of $[n]=\{1,\ldots, n\}$.  The purpose of the present paper is to undertake a study of the concept of Wilf-equivalence based on Klazar's notion.  We determine all Wilf-equivalences for partitions with exactly two blocks, one of which is a singleton block, and we conjecture that, for $n\geq 4$, these are all the Wilf-equivalences except for those arising from complementation.   If $\tau$ is a partition of $[k]$ and $\Pi_n(\tau)$ denotes the set of all partitions of $[n]$ that avoid $\tau$, we establish inequalities between $|\Pi_n(\tau_1)|$ and $|\Pi_n(\tau_2)|$ for several choices of $\tau_1$ and $\tau_2$, and we prove that if $\tau_2$ is the partition of $[k]$ with only one block, then $|\Pi_n(\tau_1)| <|\Pi_n(\tau_2)|$ for all $n>k$ and all partitions $\tau_1$ of $[k]$ with exactly two blocks.  We conjecture that this result holds for all partitions $\tau_1$ of $[k]$.  Finally, we enumerate $\Pi_n(\tau)$ for all partitions $\tau$ of $[4]$. 
\end{abstract}

\section{Introduction}
A \emph{set partition} of a set $S$ is a collection of disjoint nonempty subsets $B_1,\ldots, B_m$ of $S$ whose union is $S$.  We call the subsets $B_i$ \emph{blocks}, and we write 
$$\sigma = B_1/\cdots/ B_m \vdash S.$$
If $S$ is a set of positive integers and $\sigma\vdash S$, then the \emph{standardization} of $\sigma$ is the set partition of $\{1,2,\ldots, |S|\}$ obtained by replacing the smallest element of $S$ by $1$, the second smallest element of $S$ by $2$, and so on.  

If $n$ is a positive integer we let $[n] = \{1,2,\ldots,n\}$ and define
$$\Pi_n = \makeset{\sigma}{\sigma\vdash[n]}.$$
The concept of pattern avoidance for set partitions was introduced by Klazar in~\cite{Klazar:Count2000}.  If $k\leq n$ and we have $\sigma\vdash[n]$ and $\tau\vdash[k],$ then we say $\sigma$ \emph{contains $\tau$ as a pattern} if there is a subset $S$ of $[n]$ such that the standardization of the restriction of $\sigma$ to $S$ is $\tau$.  If $\sigma$ does not contain $\tau$, we say that $\sigma$ \emph{avoids} $\tau$.  We let 
$$\Pi_n(\tau) = \makeset{\sigma\in \Pi_n}{\sigma\mbox{ avoids } \tau}.$$

The present paper is a contribution to the study of Klazar's definition of pattern avoidance, but we should mention right away that there are other definitions.  One of these arises from the well-known correspondence  between set partitions and restricted growth functions.  Recall that a \emph{restricted growth function (RGF)}  is a word $a_1a_2\cdots a_\ell$ of positive integers such that $a_1 = 1$ and, for $i\geq 2$, we have $a_i\leq 1+ \max\{a_1,\ldots, a_{i-1}\}.$  The integer $\ell$ is called the \emph{length} of the RGF.  If $\sigma\in \Pi_n$, we define a corresponding $RGF$ of length $n$, denoted by $\rgf(\sigma)$, as follows.  First write 
$\sigma = B_1/B_2/\cdots /B_m$, where 
$$\min B_1 < \min B_2<\cdots < \min B_m.$$
(This is called the \emph{standard form of $\sigma$}.)  Then let $\rgf(\sigma) = a_1\cdots a_n$, where $i\in B_{a_i}$.  We can obtain an alternative notion of pattern avoidance for set partitions by using a natural notion of avoidance for RGF's. If $k\leq n$ and $\textbf{a}=a_1\cdots a_n$ and $\textbf{b}=b_1\cdots  b_k$ are RGFs, we say that $\textbf{a}$ \emph{contains} $\textbf{b}$ if $\textbf{a}$ has a subsequence whose standardization is $\textbf{b}$, otherwise we say that $\textbf{a}$ \emph{avoids} $\textbf{b}$.  This notion of RGF avoidance has been studied extensively (see Mansour's comprehensive book~\cite{Mansour:Combi2013}).  It does not coincide with Klazar's notion.  For if $\sigma$ avoids $\tau$ in Klazar's sense, then $RGF(\sigma)$ avoids $RGF(\tau)$; but the converse may fail.  For example $145/23$ contains $12/34$, yet $\rgf(145/23) = 12211$ avoids $\rgf(12/34) = 1122$.  

There is yet a third notion of pattern avoidance for set partitions that involves arc-diagrams (See~\cite{BloomElizalde:Pattern-2013,ChenDeng:Crossing07,RiordanThe-dist1975,TouchardSur-un-p1952}).  A well studied notion in this context is that of non-nesting and non-crossing set partitions which arise from the avoidance of certain arc-configurations.  We wish to point out that under Klazar's definition non-crossing set partitions are those that omit $13/24$.  Interestingly, there is no single pattern $\sigma$ such that the non-nesting partitions are precisely those that omit $\sigma$, in  Klazar's sense. To see this, observe that the only candidate for  $\sigma$ is $14/23$ since the non-nesting set partitions of length $4$ are $\Pi_4 \setminus \{14/23\}$.  Yet, the set partition $135/24$ is non-nesting but contains $14/23$.  

Klazar's definition has not been as well studied as the RGF definition.  The earliest paper is of course Klazar's~\cite{Klazar:Count2000}, and the most recent paper of which we are aware is the paper \cite{Dahlberg:Set-p2016} by Dahlberg, et al.  We also mention Sagan's paper~\cite{Sagan:2010aa}, which contains, along with many other results, an enumeration of $\Pi_n(\tau)$ for all $\tau\vdash[3]$.  
In Section~\ref{sec:enum} of our paper we enumerate $\Pi_n(\tau)$ for all $\tau \vdash[4]$, although our main purpose is to undertake a study of Wilf-equivalence in the context of Klazar's definition of avoidance.  If $\tau,\pi\vdash[k]$, we say that $\tau$ and $\pi$ are \emph{Wilf-equivalent}, and we write $\tau\sim \pi$, if $|\Pi_n(\tau)| = |\Pi_n(\pi)|$ for all $n>k$.  The fact that $1/2/3\sim 13/2$ is established in~\cite{Sagan:2010aa}.  In Section~\ref{sec:wilf} of our paper we establish the new Wilf-equivalences
$$12\cdots a-1,a+1\cdots k/a \sim 12\cdots b-1,b+1\cdots k/b$$
for all $k\geq 4$ and $1<a,b<k$, and based on computer evidence, we conjecture (Conjecture~\ref{conj:wilf}) that these,  $1/2/3\sim 13/2$, and $\tau\sim \tau^c$ are the only Wilf-equivalences, where $\tau^c$ is called the \emph{complement of $\tau$} and is obtained from $\tau$ by subtracting each number from $k+1$.  (There are similar conjectures for Wilf-equivalences in other contexts, all of which seem quite difficult to prove. See~\cite{AlbertBouvel:A-genera2014,KitaevLieseRemmelSagan:Rational2009}.) We also show that 
$$|\Pi_n(2\cdots k/1)| < |\Pi_n(13\cdots k/2)| \mbox{ and } |\Pi_n(1\cdots k-1/k)| < |\Pi_n(13\cdots k/2)|$$
for all $n\geq 2k-2$ but that these inequalities may become equalities when $n< 2k-2$.  For example, this occurs when $k=5$ and $n= 6,7$.  

Motivated by the results indicated in the last sentence, we introduce a partial ordering $\prec$ of $\Pi_k$, as follows.  For $\tau, \pi\in \Pi_k$ we write $\tau\prec \pi$ if $|\Pi_n(\tau)| \leq |\Pi_n(\pi)|$ for all $n>k$, and there exists some $m\geq k$ such that $|\Pi_n(\tau)| < |\Pi_n(\pi)|$ for all $n>m$.  

If $\beta_k$ denotes the partition of $[k]$ that has only one block, then computer evidence suggests that, for every $k\geq 4$ and every $\tau\neq \beta_k$ in $\Pi_k$ we have $\tau\prec \beta_k$, and in fact $|\Pi_n(\tau)|< |\Pi_n(\beta_k)|$ for all $n>k$.  In Section~\ref{sec:prec} of our paper we prove this result for all $\tau$ with exactly two blocks, and we conjecture (Conjecture~\ref{conj:order}) that the result holds for all $\tau\neq \beta_k$.  We also consider partitions $\pi$ such that there is exactly one doubleton block and all other blocks are singletons.  We prove that for every such $\pi$, if $\sigma_k$ denotes the partition all of whose blocks are singletons, then 
$$|\Pi_n(\pi)|<|\Pi_n(\sigma_k)|< |\Pi_n(\beta_k)|$$ 
for all $n>k$. This result should be compared with Klazar's result in~\cite{Klazar:Count2000} that for all such $\pi$ and $\sigma_k$ the generating function of $|\Pi_n(\pi)|$ is rational (whereas the generating function for $|\Pi_n(\beta_k)|$ is not).

\section{Wilf-equivalence}\label{sec:wilf}
For a fixed $k$, we first determine all Wilf-equivalences among  the patterns
	$$\beta_{k,a} = 1\cdots (a-1) (a+1)\cdots k/ a,$$
where  $1\leq a \leq k$.

\begin{theorem}\label{thm:wilf}
We have 
$\beta_{k,1}\sim \beta_{k,k}$.  For $n> k$ and $2\leq a \leq k-1$ we have 
$$|\Pi_n(\beta_{k,a+1})|\leq |\Pi_n(\beta_{k,a})|,$$ with equality if $a< k-1,$ and therefore $\beta_{k,a}\sim \beta_{k,b}$ when $2\leq a,b \leq k-1.$  Finally, we have 
$$|\Pi_n(\beta_{k,k})| < |\Pi_n(\beta_{k,k-1})|$$ 
when $n\geq 2k-2,$ so $\beta_{k,k} \prec \beta_{k,a}$ for $2\leq a \leq k-1.$
 
\end{theorem} 

This theorem together with computational evidence for all patterns of length $\leq 9$ suggests the following conjecture.

\begin{conjecture}\label{conj:wilf}
The only Wilf-equivalences are $1/2/3\sim 13/2$, the equivalences established by Theorem~\ref{thm:wilf}, and the equivalences resulting from complementation.  		
\end{conjecture}
 
The first assertion of Theorem~\ref{thm:wilf} is a simple consequence of complementation. To prove the second assertion, we construct an injection 
$$\phi_a:\Pi_n(\beta_{k,a+1}) \to \Pi_n(\beta_{k,a}),$$
for each $2\leq a \leq  k-1$, and show that $\phi_a$ is a bijection if $a< k-1.$  Before we can write down this mapping we need a few lemmas and observations.   We start with these immediately and delay the proof of Theorem~\ref{thm:wilf} to the end of this section.  

To begin, fix $k$ and $2\leq a\leq k-1$.  Observe that $\pi\vdash [n]$ contains $\beta_{k,a}$ if and only if  $\pi$ contains a block $B$ such that
\begin{equation}\label{eq:blocks contain}
a-1\leq |\makeset{x\in B}{x<c}|\quad\mbox{and}\quad  k-a \leq |\makeset{x\in B}{x>c}|		
\end{equation}
for some $c\notin B$.   

Consequently, we say a finite set $B$ of integers \emph{contains} $\beta_{k,a}$ if it satisfies (\ref{eq:blocks contain}), otherwise we say $B$ \emph{avoids} $\beta_{k,a}$.  Consequently, a set partition $\pi$ avoids $\beta_{k,a}$ if and only if all of its blocks avoid $\beta_{k,a}$.  

Now, consider the anatomy of a block $B$ that avoids $\beta_{k,a+1}$ but contains $\beta_{k,a}$.   Since $B$ contains $\beta_{k,a}$, there exists some $c\notin B$ so that 
$$a-1 \leq |\makeset{x\in B}{x<c}|\quad\mbox{and}\quad  k-a \leq |\makeset{x\in B}{x>c}|.$$
Furthermore, since $B$ avoids $\beta_{k,a+1}$, the set on the left must have size exactly $a-1$.  For the same reason, the set on the right must be of the form
$$y_1-\ell,\ldots,y_1-2,y_1-1,y_1<\cdots < y_{k-a-1},$$
for some $\ell\geq 1$.  Therefore the block $B$ looks like
\begin{equation}\label{eq:B decomp}
	\begin{tikzpicture}[baseline={([yshift=2ex]current bounding box.center)}]
		\draw (0,0) rectangle (2,1) node[pos=.5] {$<c$};
		\draw[fill=gray!20] (3,0) rectangle (5,1)node[pos=.5] {};
		\draw (5,0) rectangle (8,1)node[pos=.5] {};
		
		\draw [
    thick,
    decoration={
        brace,
        mirror,
        raise=0.1cm
  },decorate] (0,0) -- (2,0) 
node [pos=0.5,anchor=north,yshift=-0.2cm] {$a-1$};

		\draw [
    thick,
    decoration={
        brace,
        mirror,
        raise=0.1cm
  },decorate] (3,0) -- (5,0) 
node [pos=0.5,anchor=north,yshift=-0.2cm] {$\ell$};

\draw [
    thick,
    decoration={
        brace,
        mirror,
        raise=0.1cm
  },decorate] (5,0) -- (8,0) 
node [pos=0.5,anchor=north,yshift=-0.2cm] {$k-(a+1)$};
	\end{tikzpicture}\ \raisebox{-10pt}{,}
\end{equation} 
where the elements in the shaded block are consecutive in value and the rightmost block's smallest element is $y_1$.  By decrementing the values in the shaded region, we effectively slide  the gray rectangle as far left as possible, obtaining the new block $B'$:
\begin{equation}\label{eq:B' decomp}
	\begin{tikzpicture}[baseline={([yshift=2ex]current bounding box.center)}]
		\draw (0,0) rectangle (2,1) node[pos=.5] {$<c$};
		\draw[fill=gray!20] (2,0) rectangle (4,1)node[pos=.5] {};
		\draw (5,0) rectangle (8,1)node[pos=.5] {};
		
		\draw [
    thick,
    decoration={
        brace,
        mirror,
        raise=0.1cm
  },decorate] (0,0) -- (2,0) 
node [pos=0.5,anchor=north,yshift=-0.2cm] {$a-1$};

		\draw [
    thick,
    decoration={
        brace,
        mirror,
        raise=0.1cm
  },decorate] (2,0) -- (4,0) 
node [pos=0.5,anchor=north,yshift=-0.2cm] {$\ell$};

\draw [
    thick,
    decoration={
        brace,
        mirror,
        raise=0.1cm
  },decorate] (5,0) -- (8,0) 
node [pos=0.5,anchor=north,yshift=-0.2cm] {$k-(a+1)$};
	\end{tikzpicture}\ \raisebox{-12pt}{.}
\end{equation}

\begin{remark}\label{rmk:avoid after slide}
We remark that in $B'$, the elements in the gray block are consecutive in value and the largest element in the leftmost block is one less than the smallest element in the gray block.  Consequently, $B'$ avoids $\beta_{k,a}$ while containing $\beta_{k,a+1}$.    	
\end{remark}

The new block $B'$ gives rise to a new set partition as summarized in our next definition.
\begin{definition}
	Let $\pi\vdash [n]$ and assume the $i$th block $B$ of $\pi$ avoids $\beta_{k,a+1}$ and contains $\beta_{k,a}$.  So $B$ is as depicted in (\ref{eq:B decomp}).  Now let $B'$ be the block depicted in (\ref{eq:B' decomp}).  We define $\slide_i(\pi)\vdash [n]$ to be the set partition $\pi' = B'/ \sigma$ where $\sigma$ is the set partition of $[n]\setminus B'$ that is order-isomorphic to the set partition obtained by deleting $B$ from $\pi$.
\end{definition}

We illustrate this definition with the following short example.
\begin{example}
Let $k=5, a=2$ so that $\beta_{5,3} = 1245/3$ and $\beta_{5,2} = 1345/2$.  If	
$$\pi = 1\ 3\ / 2\ 5\ 6\ 7\ 8\ 9/ 4\ 10 \in \Pi_{10},$$
then, since its second block avoids $\beta_{5,3}$ and contains $\beta_{5,2}$, we have
$$\slide_2(\pi) = 1\ \textbf{6}\ / 2\ \textbf{3\ 4\ 5}\ 8\ 9/ \textbf{7}\ 10,$$
where the elements affected by our sliding operation are highlighted.  
\end{example}

\begin{lemma}
If $\pi$, $B$, and $B'$ are as in the above definition and $B$ is the $i$th block (in standard form) of $\pi$, then $B'$ is the $i$th block in $\slide_i(\pi)$.  
\end{lemma}
\begin{proof}
	First observe that as $a\geq 2$,  then $\min(B) = \min(B')$.  Now set
	$$[n]\setminus B = \{w_1< w_2<\cdots\}\qquad\mbox{and}\qquad [n]\setminus B' = \{w_1'< w_2'<\cdots\}.$$
	It follows from our construction of $B'$ that 
	 $$w_s < \min(B)\Longrightarrow w_s = w_s',$$ 
	 so $B'$ is no earlier than the $i$th block in $\slide_i(\pi)$ and 
	 $$w_s>\min(B)\Longrightarrow w_s\leq w_s',$$
	 so $B'$ is no later than the $i$th block in $\slide_i(\pi)$.  
\end{proof}

As the definition of our mapping $\phi_a$ involves repeated ``slide" operations, we require a lemma guaranteeing that blocks not involved in the slide operation do not change ``too much". Our next lemma spells out exactly what is meant by this.  

\begin{lemma}\label{lem:phi well defined}
Let $\pi=B_1/B_2/\cdots /B_m\vdash [n]$. Assume that for some fixed $i$, the block $B_i$ avoids $\beta_{k,a+1}$ and contains $\beta_{k,a}$ and set
$$\slide_i(\pi) = B_1'/B_2'/\cdots/B_m'.$$ 
Then for any $1\leq c\leq k$ and $j\neq i$, the block $B_j'$ avoids $\beta_{k,c}$ if and only if $B_j$ avoids $\beta_{k,c}$.  
\end{lemma}

\begin{proof}
Fix $j\neq i$, $1\leq c\leq k$ and set
$$[n]\setminus B_i = \{w_1< w_2<\cdots\}\qquad\mbox{and}\qquad [n]\setminus B_i' = \{w_1'< w_2'<\cdots\}.$$
First, we claim that  $w_t+1 = w_{t+1}$ if and only if $w'_t+1 = w_{t+1}'$.  For the moment, let us assume this claim.  Next we make the following general observation: Any block $B_j$ avoids $\beta_{k,c}$ if and only if whenever $x\in B_j$ is such that 
$$c-1 \leq |\makeset{y\in B_j}{y\leq x}|\quad\mbox{and}\quad  k-c \leq |\makeset{y\in B_j}{y\geq x}|,$$
then $x+1\in B_j$.   

As  $\pi - B_i$ is order-isomorphic to $\slide_i(\pi) - B_i'$, it now follows that $B_j$ avoids $\beta_{k,c}$ if and only if $B_j'$ does too. 

It only remains to prove our claim.  Since $B_i$ avoids $\beta_{k,a+1}$ and contains $\beta_{k,a}$ it is depicted in  (\ref{eq:B decomp}) and $B_i'$ is depicted in (\ref{eq:B' decomp}).  Observe that if $I$ is the set of values in $[n]\setminus B_i$ that fall in the gap to the left of the gray block in (\ref{eq:B decomp}) and $J$ is the set of  values in $[n]\setminus B_i'$ that fall in the gap to the right of the gray block in (\ref{eq:B' decomp}), then $I$ and $J$ are both intervals and $|I| = |J|$.  If in (\ref{eq:B decomp}) $\ell$ is the maximum of the values in the leftmost block and (for $a<k-1$) $r$ is the minimum of the values in the rightmost block, then it follows that 
\begin{equation}\label{eq:left of slide}
\makeset{w_s}{w_s<\ell} = \makeset{w_s'}{w_s'< \ell}	
\end{equation}
and
\begin{equation}\label{eq:right of slide}
\makeset{w_s}{w_s>r} = \makeset{w_s'}{w_s'> r}.
\end{equation}
Assuming  $w_t$ and $w_{t+1}$ are consecutive in value it follows that they both lie in either the set (\ref{eq:left of slide}), the set (\ref{eq:right of slide}), or in the interval $I$.  In the first two cases, it easily follows from our above equalities that   $w_t'+1=w_{t+1}'$.    In the third case, the set equality in (\ref{eq:left of slide}) along with the fact that $|I|=|J|$ implies that  $w_t',w'_{t+1}\in J$. Since $J$ is an interval, we conclude that $w_t',w'_{t+1}$ are consecutive in value.  

The proof that if $w_t'$ and $w_{t+1}'$ are consecutive in value, then $w_t$ and $w_{t+1}$ are too, is analogous.  Its details are omitted.  
\end{proof}

Finally, we are in a position to define $\phi_a$.  For any $\pi\in \Pi_n(\beta_{k,a+1})$ first set $\pi_0 = \pi$ and let $m$ be the number of blocks in $\pi$.  Having defined $\pi_i$ we obtain $\pi_{i+1}$ by considering the $(i+1)$st block of $\pi_i$.  If this block contains $\beta_{k,a}$, then set $\pi_{i+1}=\slide_{i+1}(\pi_i)$, otherwise we set $\pi_{i+1} =\pi_i$.  Lastly we set $\phi_a(\pi) = \pi_m$.  

We pause to point out that our mapping $\phi_a$ is well defined.  First, we note that Lemma~\ref{lem:phi well defined} guarantees that the $(i+1)$st block in $\pi_i$ avoids $\beta_{k,a+1}$.  (This is  crucial since if this block contained $\beta_{k,a+1}$ and also contained $\beta_{k,a}$ our sliding operation would not be defined.) Additionally, Lemma~\ref{lem:phi well defined} together with  Remark~\ref{rmk:avoid after slide} guarantees that  $\phi_a(\pi)$ avoid $\beta_{k,a}$.  

Before proving that  $\phi_a$ is a bijection we demonstrate this mapping.  

\begin{example}\label{ex:slide}
Consider $\beta_{5,3} = 1245/3$ and $\beta_{5,2} = 1345/2$ and fix
$$\pi = 1\ 10\ 11\ 12/ 2\ 4\ 5\ 8/ 3\ 6\ 7\ 9 \in \Pi_{12}(\beta_{5,3}).$$
The above algorithm for $\phi_3$ yields the following steps, where the elements affected by each slide are highlighted.  
\begin{align*}
\pi_0 =&1\ 10\ 11\ 12 / 2\ 4\ 5\ 8/ 3\ 6\ 7\ 9\\
\pi_1 =& 1\ \textbf{2}\ 11\ 12 / \textbf{3\ 5\ 6\ 9}/ \textbf{4\ 7\ 8\ 10}\\
\pi_2 =& 1\ 2\ 11\ 12 / 3\ \textbf{4}\ 6\ 9/ \textbf{5}\ 7\ 8\ 10\\
\pi_3 =& 1\ 2\ 11\ 12 / 3\ 4\ \textbf{7}\ 9/ 5\ \textbf{6}\ 8\ 10
\end{align*}

\end{example}

\begin{proof}[Proof of Theorem~\ref{thm:wilf}]

As each slide in the definition of $\phi_a$ is certainly reversible, it follows that the mapping
$$\phi_a:\Pi_n(\beta_{k,a+1}) \to \Pi_n(\beta_{k,a}),$$
is injective, provided $2\leq a \leq k-1$.  Under this restriction we therefore have 
 $$|\Pi_n(\beta_{k,a+1})|\leq   |\Pi_n(\beta_{k,a})|.$$  
 If $a<k-1,$ then $k-a\geq 2,$ so
  the composition of injective maps
	$$\Pi_n(\beta_{k,a}) \xrightarrow{\comp} \Pi_n(\beta_{k,k+1-a}) \xrightarrow{\phi_{k-a}} \Pi_n(\beta_{k,k-a}) \xrightarrow{\comp} \Pi_n(\beta_{k,a+1}),$$
	implies that 
	$$|\Pi_n(\beta_{k,a})|\leq    |\Pi_n(\beta_{k,a+1})|$$
	and proves that $\phi_a$ is bijective.  This completes the proof of the second assertion of Theorem~\ref{thm:wilf}.
	
	To prove the third assertion, we show that the injection $\phi_{k-1}:\Pi_n(\beta_{k,k})\rightarrow \Pi_n(\beta_{k,k-1})$ is not surjective when $n\geq 2k-2.$  To see this, it will be helpful to observe that if $\pi\in \Pi_n(\beta_{k,k})$ then at most one block of $\pi$ has $k-1$ or more elements.  For if one block has elements $x_1<\cdots <x_{k-1}$ and another block has elements $y_1<\cdots <y_{k-1}$, then either $x_{k-1}< y_{k-1}$ or $y_{k-1}< x_{k-1}$, so $\pi$ contains the partition $\beta_{k,k}$, a contradiction.  It now follows from the definition of $\phi_{k-1}$ that for every $\pi\in \Pi_n(\beta_{k,k})$ there is at most one block of size greater than or equal to $k-1$ in $\phi_{k-1}(\pi).$  But if $n\geq 2k-2$ then the element $1,2,\ldots,k-1/k,\ldots,2k-2/(2k-1),\ldots, n$ of $\Pi_n(\beta_{k,k-1})$ has at least two blocks of size $k-1.$    
\end{proof}

\section{Partial Ordering by $\prec$}\label{sec:prec}

\subsection{The pattern $\beta_k$}

We recall that $\beta_k$ denotes the element of $\Pi_k$ that has only one block.  If $\tau$ is any element of $\Pi_k$ other than $\beta_k$ itself, computer evidence suggests that $\tau\prec \beta_k$  and $|\Pi_n(\tau)|< |\Pi_n(\beta_k)|$ for all $n> k.$

\begin{conjecture}\label{conj:order}
Let $k\geq 4.$  If $\tau\in \Pi_k$ and $\tau\neq \beta_k$, then $\tau\prec \beta_k$ and $|\Pi_n(\tau)|< |\Pi_n(\beta_k)|$ for all $n> k.$
\end{conjecture}

It follows from Theorems~\ref{thm:twoblocks}, \ref{thm:delta_lessthan_sigma},  and \ref{thm:sigma_less_beta} that Conjecture~\ref{conj:order} is true for $k=4$.  

\begin{theorem}\label{thm:twoblocks} Let $k\geq 4$.  If $\sigma \in \Pi_k$ and $ \sigma$ has exactly two blocks, then $\sigma \prec \beta_k.$  In fact, $|\Pi_n(\sigma)| < |\Pi_n(\beta_k)|$ for all $n>k.$
\end{theorem} 
\begin{proof} Suppose $\sigma\in \Pi_k$ and $\sigma$ has exactly two blocks, so $\sigma=A/B,$ with $1\in A.$ Suppose $a_1,b_1,a_2,b_2,\ldots,a_j,b_j$ are positive integers and $a_{j+1}$ is a nonnegative integer such that the first $a_1$ elements of $[k]$ are in $A$, the next $b_1$ elements of $[k]$ are in $B$, the next $a_2$ elements of $[k]$ are in $A$, and so on, so that $a_{j+1}=0$ if $k\in B$ and $a_{j+1}> 0$ if $k\in A.$
 
To prove the theorem we establish, for any given $n>k,$ a nonsurjective injection $$\varphi:\Pi_n-\Pi_n(\beta_k)\rightarrow \Pi_n-\Pi_n(\sigma).$$ To define $\varphi,$ take any $\pi\in \Pi_n-\Pi_n(\beta_k).$  If $\pi\in \Pi_n-\Pi_n(\sigma),$ let $\varphi(\pi)=\pi.$

Now suppose $\pi\in \Pi_n(\sigma)$. Since $\pi\in \Pi_n-\Pi_n(\beta_k)$, $\pi$ has at least one block with at least $k$ elements.  We obtain $\varphi(\pi)$ from $\pi$ by partitioning each such block $C$ in the following way.

Recalling that $\sigma=A/B$, let $|B|=m$, so that $m=b_1+\cdots +b_j$.  Write $$|C|-|A|=qm+r,$$ where $q\geq 1$ and $0\leq r <m$ are integers.  Then we have 
$$|C|=|A|+r+q(b_1+\cdots +b_j).$$ 
We order the elements of $C$ from smallest to largest and define a subset $A^*$ of $C$ of cardinality $|A|+r$, as follows.  Let the first $a_1$ elements of $A^*$ be the first $a_1$ elements of $C$. Skip over the next $qb_1$ elements of $C$, and let the next $a_2$ elements of $C$ be the next $a_2$ elements of $A^*.$  Then skip over the next $qb_2$ elements of $C$ and let the next $a_3$ elements of $C$ be the next $a_3$ elements of $A^*.$ Continuing in this way, define the first $a_1+\cdots +a_{j+1}=|A|$ elements of $A^*.$ Then add the last $r$ elements of $C$ to $A^*,$ so that $|A^*|=|A|+r.$

Next we define subsets $B_1,\ldots,B_q$ of $C$ such that $|B_i|=m$ for $1\leq i\leq q$.  Take the first $qb_1$ elements of $C$ that were skipped over in the construction of $A^*$, and put the first $b_1$ of these elements in $B_1,$ the next $b_1$ of these elements in $B_2,$ and so on. Then take the next $qb_2$ elements of $C$ that were skipped over in the construction of $A^*$ and put the first $b_2$ of these elements in $B_1,$ the next $b_2$ of these elements in $B_2,$ and so on.  Continuing in this way, complete the construction of $B_1,\ldots,B_q$.

We partition $C$ into blocks $A^*,B_1,\ldots,B_q$.  By construction, we see that for each $B_i$, the partition $A^*/B_i$ contains the partition $\sigma$.  We have $|B_i|=m<k$ and $|A^*|=|A|+r< |A|+m=k$.

When $\pi\in \Pi_n(\sigma)$, we obtain $\varphi(\pi)$ from $\pi$ by partitioning each block $C$ of size at least $k$ in the way just indicated.  Clearly, $\varphi(\pi)\in \Pi_n-\Pi_n(\sigma)$.  We also have $\varphi(\pi)\in \Pi_n(\beta_k)$, by the last sentence of the preceding paragraph.  For $\pi\in \Pi_n-\Pi_n(\sigma)$ we had $\varphi(\pi)=\pi\notin\Pi_n(\beta_k)$, so to show that $\varphi$ is one-to-one it suffices to show that, for $\pi\in \Pi_n(\sigma)$, we can recover $\pi$ from $\varphi(\pi)$.

To show this, observe that each block of $\varphi(\pi)$ is a subset of a block of $\pi$, and that if $D,E$ are blocks of $\varphi(\pi)$ such that the partition $D/E$ contains $\sigma$, then since $\pi\in\Pi_n(\sigma)$, $D$ and $E$ must be subsets of the same block of $\pi$.  Thus we obtain $\pi$ from $\varphi(\pi)$ by coalescing into one block the elements of any blocks $D,E$ of $\varphi(\pi)$ such that $D/E$ contains $\sigma$.

Finally, to show that $\varphi$ is not surjective, first suppose that $m+1< k$ and consider the partition $\gamma$ of $[n]$ whose blocks are $A, B\cup\{k+1\}$ and $n-k-1$ singleton blocks. Then $\gamma\in \Pi_n-\Pi_n(\sigma)$.  Suppose $\pi\in \Pi_n-\Pi_n(\beta_k)$ and $\varphi(\pi)=\gamma$.  Since $m+1< k$, we have $\gamma\in \Pi_n(\beta_k)$, so $\gamma \neq \pi$ and thus $\pi\in \Pi_n(\sigma)$.   Since $\varphi(\pi)=\gamma,$ it follows from the definition of $\varphi$ that $A\cup B \cup \{k+1\}$ is contained in some block $F$ of $\pi,$ and $|F|> k$.   In defining $\varphi(\pi)$, the block $A^*$ derived from $F$ contains the smallest element of $F$, namely 1, so $A^*=A$ (since $A$ is the block of $\gamma$ that contains 1). The other blocks derived from $F$ have $m$ elements each.  This is impossible, since the block $B\cup \{k+1\}$ is derived from $F$.  

Now consider the case $m+1=k$.  In this case $\sigma=\beta_{k,1}$.  Since $\beta_{k,1}\sim\beta_{k,k}$ by Theorem~\ref{thm:wilf} and $\varphi$ is not surjective when $\sigma=\beta_{k,k}$ by the preceding paragraph, it follows that $\varphi$ is not surjective when $\sigma=\beta_{k,1}$.\end{proof}

\subsection{The pattern $\sigma_k$}

We recall that $\sigma_k = 1/2/3/\cdots/ k$.  In this subsection we prove that for all $\delta\in \Pi_k$,  whose blocks consist of all singletons except for one doubleton block,
$$\delta\prec \sigma_k\prec \beta_k.$$ 
We begin with a lemma.

\begin{lemma}\label{lem:induction} Let $\alpha\vdash [k-1]$.   If 
$$|\Pi_n(\alpha)| <|\Pi_n(\sigma_{k-1})|$$ 
for all $n>k-1$, then  
$$|\Pi_n(1/\alpha')| <|\Pi_n(\sigma_{k})|\quad\mbox{and}\quad |\Pi_n(\alpha/k)| <|\Pi_n(\sigma_{k})|, $$ 
for all $n>k$, where $\alpha'$ is obtained by incrementing all the values in $\alpha$ by $1$.  
\end{lemma}

\begin{proof}
Proving the inequality on the left is sufficient, since the inequality on the right then follows by complementation.   

To do this, fix $n>k$ and assume there exists nonsurjective injections $\phi_m :\Pi_m(\alpha)\to \Pi_m(\sigma_{k-1})$ for all $m> k-1$. Observe that for any $B_1/B_2/\cdots/ B_m\in \Pi_n(1/ \alpha')$, in standard form,  the set partition $B_2/\cdots /B_m$ avoids $\alpha$ as $1\in B_1$.  Therefore we obtain a nonsurjective injection $\psi$ from $\Pi_n(1/\alpha')$ to $\Pi_n(\sigma_k)$ as follows: If $n-|B_1| > k-1$ we let
	$$\psi(B_1/B_2/ \cdots/ B_m) =  B_1/ \phi_{n-|B_1|}(B_2 / \cdots/ B_m).$$
	(To be precise $\phi_{n-|B_1|}(B_2 / \cdots/ B_m)$ is obtained by first standardizing the partition $B_2/\cdots/B_m$, then applying $\phi_{n-|B_1|}$, and then incrementing the values.) If $n-|B_1| \leq k-1$ we let 
	$$\psi(B_1/B_2/ \cdots/ B_m) = B_1/B_2/ \cdots/ B_m,$$
	unless $B_2/ \cdots/ B_m$ consists of exactly $k-1$ singleton blocks whose union is $[n]\setminus B_1$.  In this exceptional case we set $\psi(B_1/B_2/ \cdots/ B_m)$ to be the partition $B_1/A$ where $A$ is the partition of $[n]\setminus B_1$ whose standardization is $\alpha$. 

	The existence of this injection proves our claim.

\end{proof}

To prove our next theorem, we first recall the bijection between restricted growth functions and set partitions as described in the Introduction.   Next, define $\rgf_n^{<k}$ to be the set of all $w\in \rgf$ in the letters $\{1,\ldots, k-1\}$ with length $n$.  It immediately follows that the standard bijection between $\rgf$ and set partitions restricts to a bijection between $\rgf_n^{<k}$ and $\Pi_n(\sigma_k)$.  

\begin{theorem}\label{thm:delta_lessthan_sigma}
	Let $\delta$ be a pattern of length $k\geq 4$ that consists of all singletons except for exactly one doubleton.  Then 
	$$|\Pi_n(\delta)| < |\Pi_n(\sigma_k)|$$
	for all $n>k\geq 3$.  Consequently, $\delta\prec \sigma_k$.  
\end{theorem}
\begin{proof}
If $a$ and $b$ are the elements of the doubleton block of $\delta$, we consider cases depending on the value of $|a-b|$.  

If $|a-b|=1$, then since $k\geq 4$ we see by using Lemma~\ref{lem:induction} that it suffices to show that $|\Pi_n(\alpha)| < |\Pi_n(\sigma_4)|$ where $n> 4$ and $\alpha$ is one of $12/3/4$, $1/23/4$, or $1/2/34$.  By Lemma~\ref{lem:induction} again, it then suffices to show that $|\Pi_n(\beta)|< |\Pi_n(\sigma_3)|$ where $n> 3$ and $\beta$ is one of $12/3$ or $1/23$.  For either choice of $\beta$ we have $|\Pi_n(\beta)| = 1+ \binom{n}{2}$ and $|\Pi_n(\sigma_3)| = 2^{n-1}$ by Theorem 2.5 of (\cite{Sagan:2010aa}), and this concludes the proof when $|a-b| =1$.  

If $|a-b|=2$, then we see by using Lemma~\ref{lem:induction} that it suffices to show that $|\Pi_n(\alpha)|< |\Pi_n(\sigma_4)|$ when $n>4$ and $\alpha$ is one of $13/2/4$ or $1/3/24$.  Since $13/2/4$ is obtained from $1/3/24$ by complementation it suffices to deal with $1/3/24$.  We do so in Lemma~\ref{lem:1_24_3_less_singleton}, following the enumeration of $|\Pi_n(1/3/24)|$.  

To deal with the case $|a-b|\geq 3$, it suffices, by Lemma~\ref{lem:induction}, to prove the result for $\delta = 1k/2/3/\cdots / k-1$.  For such a $\delta$, we see that every set partition in $\Pi_n(\delta)$ is obtained from a unique set partition $B_1/\cdots /B_m\in \Pi_{n-1}(\delta)$ by inserting $n$ into one of the $k-2$ rightmost blocks $B_{m-(k-3)}, \ldots, B_m$ or by inserting $n$ into a new singleton block.    We encode these insertion choices as follows:
\begin{center}
$\begin{array}{ccl}
	1&\leftrightarrow & \mbox{ into a new singleton block}\\
	2&\leftrightarrow & \mbox{ into } B_m\\
	3&\leftrightarrow & \mbox{ into } B_{m-1}\\
	&\vdots &\\
	k-1&\leftrightarrow & \mbox{ into } B_{m-(k-3)}.
\end{array}$	
\end{center}	
Now define the set $R_n^{<k}$ to be the set of all words $w$ in the letters $1,2,\ldots, k-1$ such that $w_1 = 1$ and 
$$w_s \leq 1+ \#\mbox{ of 1's in the subword } w_1\cdots w_{s-1}.$$
Using the above encoding, it is clear that we have an injective mapping from $\Pi_n(\delta)$ into the set $R_n^{<k}$.  This mapping is not surjective for $n>k\geq 4$ since then the set partition  
$$1(k-1)/2(k+1)/3/\cdots /k-2/k/k+2/\cdots / n\notin \Pi_{n}(\delta)$$ 
is mapped, under the above encoding, to the word 
$$\underbrace{11\cdots 1}_{k-2} (k-1) 1 (k-1)\underbrace{11\cdots 1}_{n-k-1} \in R_{k+1}^{<k}.$$

It now suffices to prove that $|\rgf_n^{<k}|=|R_n^{<k}|$ since $|\rgf_n^{<k}| = |\Pi_n(\sigma_k)|$ as mentioned above.  To see this consider a word $w\in \rgf_n^{<k}$ and decompose it according to the first occurrence, from left to right, of each letter.  Doing so $w$ decomposes into the subwords
\begin{center}
\begin{tikzpicture}
		\draw (0,0)rectangle (.5,.5) node[pos=.5] {\large \textbf{1}};
		\draw (.6,0) rectangle (3,.5) node[pos = .5] {1's};
		\draw (3.2,0)rectangle (3.7,.5) node[pos=.5] {\large \textbf{2}};
		\draw (3.9,0) rectangle (6.6,.5) node[pos = .5] {1,2's};
		\draw (6.8,0)rectangle (7.3,.5) node[pos=.5] {\large \textbf{3}};
		\draw (7.5,0) rectangle (10,.5) node[pos = .5] {1,2,3's};
		\node at (10.7,.2) {$\cdots$};
		
		\draw (11.5,0)rectangle (12,.5) node[pos=.5] {\large \textbf{m}};
		\draw (12.2,0) rectangle (14.7,.5) node[pos = .5] {$1,2,\ldots, m's$};
		
		\draw [thick, decoration={
        brace,
        mirror,
        raise=0.1cm}, decorate] (.6,0) -- (3,0) 
node [pos=0.5,anchor=north,yshift=-0.2cm] {$u_1$};

\draw [thick, decoration={
        brace,
        mirror,
        raise=0.1cm}, decorate] (3.9,0) -- (6.6,0) 
node [pos=0.5,anchor=north,yshift=-0.2cm] {$u_2$};

\draw [thick, decoration={
        brace,
        mirror,
        raise=0.1cm}, decorate] (7.5,0) -- (10,0) 
node [pos=0.5,anchor=north,yshift=-0.2cm] {$u_3$};

\draw [thick, decoration={
        brace,
        mirror,
        raise=0.1cm}, decorate] (12.2,0) -- (14.7,0) 
node [pos=0.5,anchor=north,yshift=-0.2cm] {$u_m$};
		
	\end{tikzpicture}\ \raisebox{18pt}{.}
\end{center} 

As $w\in\rgf_n^{<k}$ we know that $m\leq k-1$.  In the case that $m<k-1$ we define $w'$ to be the new word given by incrementing each of the subwords $u_i$ by 1 and replacing each letter's first occurrence by a 1.  So $w'$ decomposes, according to its occurrences of 1's, as 
\begin{center}
\begin{tikzpicture}
		\draw (0,0)rectangle (.5,.5) node[pos=.5] {\large \textbf{1}};
		\draw (.6,0) rectangle (3,.5) node[pos = .5] {$u_1+1$};
		\draw (3.2,0)rectangle (3.7,.5) node[pos=.5] {\large \textbf{1}};
		\draw (3.9,0) rectangle (6.6,.5) node[pos = .5] {$u_2+1$};
		\draw (6.8,0)rectangle (7.3,.5) node[pos=.5] {\large \textbf{1}};
		\draw (7.5,0) rectangle (10,.5) node[pos = .5] {$u_3+1$};
		\node at (10.7,.2) {$\cdots$};
		\draw (11.5,0)rectangle (12,.5) node[pos=.5] {\large \textbf{1}};
		\draw (12.2,0) rectangle (14.7,.5) node[pos = .5] {$u_m+1$};
	\end{tikzpicture}\ \raisebox{1pt}{.}
\end{center} 
It easily follows that $w'\in R_n^{<k}$ since the letters in $u_i+1$  are at most $i+1$ and are preceded by exactly $i$ occurrences of the letter 1. 

In the case that $m=k-1$, we modify this mapping only slightly.  Here we define $w'$ to be the word
\begin{center}
\begin{tikzpicture}
		\draw (0,0)rectangle (.4,.5) node[pos=.5] {\large \textbf{1}};
		\draw (.6,0) rectangle (2.4,.5) node[pos = .5] {$u_1+1$};

		\draw (2.6,0)rectangle (2.9,.5) node[pos=.5] {\large \textbf{1}};
		\draw (3.1,0) rectangle (5,.5) node[pos = .5] {$u_2+1$};
		
		\node at (7,.2) {$\cdots$};
		
		\draw (8.5,0) rectangle (8.9,.5) node[pos=.5] {\large \textbf{1}};
		\draw (9.1,0) rectangle (11.4,.5) node[pos = .5] {$u_{k-2}+1$};

		\draw (11.6,0)rectangle (11.9,.5) node[pos=.5] {\large \textbf{1}};
		\draw (12.1,0) rectangle (14,.5) node[pos = .5] {$u_{k-1}$};
	\end{tikzpicture}\ \raisebox{1pt}{,}
\end{center} 
so that all subwords, except the last one, $u_{k-1}$, are incremented by 1. Note this second case is needed since the letter $k-1$ might occur in $u_{k-1}$ and the letter $k$ is not available. Again it is clear that $w'\in R_n^{<k}$.  In this way we obtain an injective mapping from $\rgf_n^{<k}$ to $R_n^{<k}$.  Furthermore, this mapping is easily seen to be surjective as we can decompose any word $v$ in $R_n^{<k}$  according to its first $k-1$ ones and then reverse the above mapping.  It only remains to prove that applying the reverse mapping to any $v\in R_n^{<k}$ results in a word $w\in \rgf_n^{<k}$.  
This easily follows from the observation that if $v_s$ is a letter in $v$ with $j$ ones to its left, then 
$$v_s\leq j+1 = \max(w_1,\ldots, w_{s-1}) +1.$$
First assume  $j<k-1$.  In this case, if $v_s\neq 1$, then $w_s=v_s-1$ and if $v_s =1$, then $w_s = j+1$.  Either way, such letters satisfy the condition of a restricted growth function.  The case when $j\geq k-1$ is similar.  Its details are left to the reader.

\end{proof}

\begin{theorem}\label{thm:sigma_less_beta}
	For all $k\geq 2$ and $n\geq 1$
	$$|\Pi_n(\sigma_k)| \leq |\Pi_n(\beta_k)|.$$ 
	Moreover, this inequality is strict provided $n>k>2$.  
\end{theorem}

\begin{proof}
First, we show, by induction on $k$, that there exists a family of injections
	$$\psi_{k,n}: \Pi_n(\sigma_k) \to \Pi_n(\beta_k),$$
	for all $k\geq 2$ and $n\geq 1$. 
To streamline this notation we drop the subscript $n$ and write $\psi_k$ instead of $\psi_{k,n}$.  

To begin our induction argument note that $|\Pi_n(\sigma_2)| =1=|\Pi_n(\beta_2)|$. Now, consider $\pi = B_1/\cdots/B_m\in \Pi_n(\sigma_{k+1})$  and observe that  $B_2/\cdots/B_m$ can be considered a set partition of $[n-|B_1|]$ that avoids $\sigma_k$.  Therefore $\psi_k(B_2/\cdots/B_m)$ is well defined, avoids $\beta_k$, and has blocks of size $<k$.  (If $B_2/\cdots/B_m = \emptyset$ we set $\psi_k(\emptyset) = \emptyset$.) With this in mind, the division algorithm yields 
$$|B_1| = q\cdot k + r,$$ 
where $0\leq r\leq k-1$.  Provided $r>0$ let $C_0/C_1/C_2/\cdots/C_q$ be the set partition of $B_1$ where $C_0$ consists of the first $r$ numbers of $B_1$ and $C_1$ consists of the next $k$ numbers, etc.  Note that in the case  $r=0$ we simply ignore $C_0$.   Finally, we define 
	$$\psi_{k+1}(\pi)  = C_0/C_1/ \cdots/C_q/\psi_k(B_2/\cdots/B_m),$$
so that $C_1,\ldots, C_q$ are the only blocks of size $k$ in this partition.  Note that if $r>0$, then $1\in C_0$, otherwise $1\in C_1$.  Consequently, $\psi_{k+1}$ is injective since $\psi_k$ is injective and $B_1$ may be recovered by taking the union of all blocks of size $k$ in $\psi_{k+1}(\pi)$ along with the block containing 1.
	
It only remains to show that our mappings $\psi_{n,k}$ are not surjective when $n>k>2$.	To see this observe that in the construction of $\psi_{n,k}(\pi)$ the first block of $\pi$ (in standard form) is partitioned into consecutive segments as described above.  As a result, no partition in the image of $\psi_{n,k}$  could contain the blocks:
$$1\ 3/  2\ 4\ 5\ \cdots (k+1),$$
because both blocks would have to come from $B_1$, but this violates consecutiveness. Hence these mappings cannot be surjective.
\end{proof}

\begin{corollary}
	Let $\delta$ be any pattern of length $k$ consisting of all singletons except for one doubleton.  Then, 
	$$|\Pi_n(\delta)| < |\Pi_n(\sigma_k)| <|\Pi_n(\beta_k)|,$$
	for all $n>k\geq 4$.
\end{corollary}

The proof of this corollary is a direct consequence of the previous two theorems.

\section{Enumeration}\label{sec:enum}
In this section we concentrate on the enumeration of $|\Pi_n(\tau)|$ for patterns $\tau\vdash [4]$.  As enumerations for the patterns $1/2/3/\cdots/k$ and $\beta_k=12\cdots k$ are well known using exponential generating functions and Sagan in~\cite{Sagan:2010aa} enumerated the general pattern $12/3/4/\cdots/k$, we concentrate on all others with one exception.  As the enumeration for the pattern $1/23/4$  devolves into  numerous (uninteresting) cases and Klazar in~\cite{Klazar:Count2000} showed that the resulting generating function is rational, we choose to omit this pattern from our discussion.  

To summarize the enumerations established in the following subsections (as well as those mentioned above) we include the following table of results where $\displaystyle\exp_m =\sum_{i=0}^m\frac{x^i}{i!}$ and 
$S(n,k)$ denotes the Stirling numbers of the second kind and  
$\displaystyle G(z) = \frac{z-2z^2(1+z)-z\sqrt{1-4z^2}}{-2+2z(1+z)^2}$. Aside from the omitted pattern $1/23/4$, every pattern of length 4 is, up to complementation, included in the table.
\begin{center}
{\tabulinesep=1.4mm
\begin{tabu}{c|c|c}
	Pattern&Enumeration&Reference\\
	\hline
	\hline
	$1234$ & $\displaystyle\exp(\exp_{3}(z)-1)$& -\\

	$1/2/3/4$ & $\displaystyle\exp_{3}(e^z-1)$& -\\
	$12/3/4$&$\displaystyle 1+\sum_{k=1}^{n-1}\sum_{j=1}^{m-2}S(n-k,j)\sum_{i=1}^j \binom{j-1}{i-1}(k)_i$&\cite{Sagan:2010aa}\\
		$12/34$&$\displaystyle\sum_{k=0}^{\lfloor n/2\rfloor} k!{n \choose 2k}+\sum_{\ell=3}^{n-2k}{n\choose 2k+\ell}k!(k+1)^2$&(\ref{sec:12_34})\\

	$1/234$&$\displaystyle \sum_{\ell=1}^{n} M(n-\ell) + \sum_{\ell=1}^{n-2} (n-\ell-1)M(n-\ell-1)+\sum_{\ell=1}^{n-3} {n-\ell-1 \choose 2}M(n-\ell-2)$ &(\ref{sec:1_234})\\
	$134/2$&$\displaystyle 1+ \sum_{k=1}^{\lfloor n/2 \rfloor}\sum_{f=0}^{n-2k} {n-f-k-1\choose k-1}{2k+f\choose f}(2k)!!$&(\ref{sec:134_2})\\

	$14/23$&$\displaystyle G\left(\frac{z}{1-z}\right)\frac{1}{1-z} + \frac{1}{1-z}$&(\ref{sec:14_23})\\
	$13/24$&$\displaystyle \frac{1-\sqrt{1-4z}}{2z}$&-\\
	
	$14/2/3$ & $\displaystyle\frac{z-3z^2+3z^3}{1-5z+8z^2-5z^3}$& (\ref{sec:14_2_3})\\

	$1/24/3$&$\displaystyle \frac{z-4z^2+6z^3-2z^4}{(1-3z+2z^2)^2}$&(\ref{sec:1_24_3})

\end{tabu} }
\end{center}

\subsection{The pattern $14/2/3$}\label{sec:14_2_3}
Let us begin by recalling the standard recursive construction for building set partitions.  That is, every set partition in $\Pi_n$ is obtained by taking a set partition in $\Pi_{n-1}$ and either adding $n$ to an existing block or appending $n$ as a singleton.  To enumerate $\Pi_n(14/2/3)$ we consider a refinement of this recursive construction.  First, observe that if $\pi \in \Pi_n(14/2/3)$, where in standard form $\pi = B_1/\cdots /B_m$, we have either $n\in B_{m-1}$ or $n\in B_m$, as any other choice would force an occurrence of the pattern $14/2/3$.  From this observation we may immediately restrict our attention to set partitions which are built by recursively placing the next largest element into either a singleton block, the last block, or the second to last block.  To formalize this refinement let us define $W_n$ to be the subset of all words in the letters \textbf{a}, \textbf{b}, \textbf{c}, that start with an \textbf{a} and have the property that any occurrence of the letter \textbf{c} must be preceded by at least two \textbf{a}'s. Using this we obtain an injection
$$\varphi:W_n\to \Pi_n$$
which is defined recursively as follows. First, set $\varphi(\textbf{a}) = 1$.  Then, for any $w\in W_n$, define $\varphi(w)$ to be the set partition obtained from $\pi'=\varphi(w_1\cdots w_{n-1})$ by appending $n$ as a singleton if $w_n = \textbf{a}$, inserting $n$ into the (existing) rightmost block  of $\pi'$ if $w_n = \textbf{b}$, or inserting $n$ into the second to last block of $\pi'$ if $w_n = \textbf{c}$.  Note that our insistence that words in $W_n$ have the property that any \textbf{c} is preceded by at least two \textbf{a}'s guarantees that if $w_n=\textbf{c}$, then $\pi'$ contains at least two blocks.  

It is clear for our definition that $\varphi$ is injective.  Additionally, it follows from our first observation in this subsection that 
$$\Pi_n(14/2/3)\subseteq \varphi(W_n).$$
Before continuing we next provide an example of this construction in Example~\ref{ex:phi14_2_3}.

\begin{example}\label{ex:phi14_2_3}
	If $w = \textbf{aaccba}$, then 
	$$\varphi(w) = 134/25/6,$$ 
	and if $v = \textbf{aacabc}$, then  
		$$\varphi(v) = 13/26/45$$
	which is not $14/2/3$-avoiding.  
\end{example}
An immediate consequence of Example~\ref{ex:phi14_2_3}, is that $\Pi_n(14/2/3)\subsetneq \varphi(W_n)$. Consequently, we seek a subset of $W_n$ whose image under $\varphi$ is precisely $\Pi_n(14/2/3)$. To this end consider the subset $W_n^*$ consisting of all $w\in W_n$ with the property that the letters between any two \textbf{c}'s do not contain exactly one \textbf{a}.  Observe that in Example~\ref{ex:phi14_2_3} $w\in W_n^*$ but $v$ is not.  

\begin{lemma}
	The restricted mapping $\varphi:W_n^*\to \Pi_n(14/2/3)$ is a bijection.  
\end{lemma}

\begin{proof}
	We first show that $\Pi_n(14/2/3)\subseteq\varphi(W_n^*)$.  To this end consider $w\in W_n\setminus W_n^*$ and let $i<j<k$ be such that $w_i=w_k=\textbf{c}$ and $w_j=\textbf{a}$ is the only occurrence of the letter $\textbf{a}$ between these two \textbf{c}'s. Next, set 
	$$\varphi(w) = B_1/\cdots /B_m$$
	so that $i\in B_t$, for some $t$.  As $w_i=c$ that means that $\ell:=\min(B_{t+1})< i$. As there is no \textbf{a} between $w_i$ and $w_j$ we see that $j\in B_{t+2}$.  Moreover, as there is no \textbf{a} between $w_j$ and $w_k$ we see that $k\in B_{t+1}$.  This provides our desired contradiction since the integers $\ell<i<j<k$ in the blocks $B_t, B_{t+1}$, and $B_{t+2}$ create an occurrence of the forbidden pattern $14/2/3$.
	
	To establish the other inclusion, consider $w\in W_n^*$ and assume for a contradiction that $\varphi(w)\notin\Pi_n(14/2/3)$.  Again set
		$$\varphi(w) = B_1/\cdots /B_m.$$
		By definition of the map $\varphi$, we see that $\min(B_i) > \max(B_j)$ for all $j+1<i$.  From this it follows that any occurrence of $14/2/3$ in $\varphi(w)$ must occur among three consecutive blocks $B_t, B_{t+1}, B_{t+2}$ and involve integers $\ell<i<j<k$ such that $\ell, k\in B_{t+1}$, $i\in B_t$, and $j\in B_{t+2}$.  This immediately implies that $w_i = w_k = \textbf{c}$ and that $w_j= \textbf{a}$ is the only \textbf{a}  between these two \textbf{c}'s.  This contradiction permits us to conclude that $\varphi(W_n^*) = \Pi_n(14/2/3)$ as claimed.    
	   
\end{proof}

\begin{theorem}\label{thm:enum14_2_3}
We have 
$$F_{14/2/3}(z) = \sum_{n\geq 1} |\Pi_n(14/2/3)|z^n = \frac{z-3z^2+3z^3}{1-5z+8z^2-5z^3}.$$
\end{theorem}
\begin{proof}
	By our previous lemma it suffices to find $\sum_{n\geq 1} |W_n^*|z^n$.  To this end we consider three cases depending on how many \textbf{c}'s our word contains.  
	
	The words in $W_n^*$ with no \textbf{c}'s are easily counted by the expression $\frac{z}{1-2z}$.  On the other hand, the words with exactly one \textbf{c}'s are counted by 
	$$z^2\left(\frac{1}{1-2z}-\frac{1}{1-z}\right)\frac{1}{1-2z}$$
	where the second term follows since any occurrence of a \textbf{c} must be preceeded by at least two \textbf{a}'s.  Lastly, we consider  words containing at least two \textbf{c}'s.  It is clear from the definitions that such words decompose as 

	$$\textbf{a}\ \underbrace{\boxed{\textbf{b}'s\mbox{ and }\textbf{a}'s}}_{\mbox{at least one } \textbf{a}}\ \textbf{c}\cdots\textbf{c}\cdots\textbf{c}\cdots\textbf{c}\ \boxed{\textbf{b}'s\mbox{ and }\textbf{a}'s},$$ 
	so that between any two consecutive \textbf{c}'s we do not have exactly one \textbf{a}.  Counting the words between any two consecutive \textbf{c}'s we have 
	$$G(z) = \frac{1}{1-2z} - \sum_{i\geq 1} iz^i = \frac{1}{1-2z} - \frac{z}{(1-z)^2}$$
	since such words are in the letters \textbf{a} and \textbf{b} but cannot have exactly one \textbf{a}.  In terms of $G(z)$ we see that 
	$$z\left(\frac{1}{1-2z}-\frac{1}{1-z}\right)\left(\frac{z^2G(z)}{1-zG(z)}\right)\left(\frac{1}{1-2z}\right)$$
	counts the case where our words have at least two \textbf{c}'s.  Summing these three cases yields the desired result.  
\end{proof}

\subsection{The pattern $1/24/3$}\label{sec:1_24_3}
The enumeration of the pattern $1/24/3$ closely resembles that of the enumeration of $14/2/3$ found in the previous section.  As a result, we begin similarly by observing that if $\pi=B_1/\cdots/B_m\in \Pi_n(1/24/3)$ then we can only have $n\in B_1$ or $n\in B_m$, as any other choice creates our forbidden pattern.  Recalling the set $W_n$ defined in the previous subsection, we define (recursively) the function 
$$\phi:W_n\to \Pi_n.$$  
First set $\phi(\textbf{a}) = 1$.  Next, for any $w\in W_n$, define $\phi(w)$ to be the set partition obtained from $\phi(w_1\cdots w_{n-1})=B_1/\cdots /B_m$ by inserting $n$ into a singleton block if $w_n = \textbf{a}$, inserting $n$ into the block $B_m$ if $w_n = \textbf{b}$, and lastly, inserting $n$ into the block $B_1$ provided $w_n=\textbf{c}$.  It now follows, since the words in $W_n$ have the property that any \textbf{c} must be preceded by at least two \textbf{a}'s, that $\phi$ is injective.  It also follows, from our initial observation, that $\Pi_n(1/24/3)\subseteq \phi(W_n)$.  
\begin{example}\label{ex:1_24_3}
If $w = \textbf{abbacb}$, then 
$$\phi(w) = 1235/46,$$	
and if $v = \textbf{aacac}$, then 
$$\phi(v) = 135/2/4.$$	
Note that $\phi(v)\notin\Pi_5(1/24/3)$.  
\end{example}
We see from this example that $\Pi_n(1/24/3)\subsetneq \phi(W_n)$.  As in the previous section, we seek a subset of $W_n$ whose image under $\phi$ is precisely $\Pi_n(1/24/3)$.  To this end, define $W_n^{**}$ to be the set of all $w\in W_n$ such that
\begin{enumerate}
	\item[1)] no \textbf{a} falls between any two \textbf{c}'s in $w$, and
	\item[2)] any \textbf{c} in $w$ which is preceded by at least three \textbf{a}'s cannot be immediately followed by a \textbf{b}.
\end{enumerate}

In the next lemma we prove that $W_n^{**}$ is this desired set.  To facilitate the reading of its proof we pause to highlight a couple key observations. Setting $w\in W_n$ with $\phi(w) = B_1/\cdots/B_m$, we first see that $\max(B_s)<\min(B_{s+1})$ for all $s>1$.   Our second observation is that  $w_i = \textbf{a}$ is the $t$th \textbf{a} in our word (from left to right) if and only if $\min(B_t) = i$.  With these observations in mind we now state and prove our lemma.  

\begin{lemma}
	The restricted map $\phi:W_n^{**} \to \Pi_n(1/24/3)$ is a bijection.  
\end{lemma}
\begin{proof}
We first show that $\Pi_n(1/24/3)\subseteq \phi(W_n^{**})$.  To do so, it suffices to show that if $w\in W_n\setminus W_n^{**}$, then $\phi(w)\notin \Pi_n(1/24/3)$.  Begin by setting $\phi(w) = B_1/\cdots/ B_m$.  We address  the two ways in which $w$ can fail to be a member of $W_n^{**}$ separately.  First, assume condition 1) fails, and let $i<j<k$ be such that $w_i=w_k=\textbf{c}$ and $w_j = \textbf{a}$.  In particular, $i,k\in B_1$.  Moreover, as any \textbf{c} must be preceded by at least two \textbf{a}'s, then it follows that $\min(B_2)<i<k$.  It also follows that if $w_j$ is the $t$th \textbf{a} in $w$, then $t\geq 3$.  Consequently, the blocks $B_1, B_2$, and $B_t$ contain an occurrence of $1/24/3$.   
	
	Next, let us assume condition 2) fails.  Here we assume that there exists some index $i$ so that $w_iw_{i+1}=\textbf{cb}$ and $w_i$ is preceded by  at least three \textbf{a}'s.  It immediately follows that $\phi(w_1\cdots w_{i-1})$ has $t\geq 3$ blocks and that in $\phi(w)$, $i\in B_1$, $i+1\in B_t$ and $\min(B_2)<\min(B_t)<i<i+1$.  Consequently, the blocks $B_1, B_2$, and $B_t$, contain an occurrence of our forbidden pattern.  We conclude that $\Pi_n(1/24/3)\subseteq \phi(W_n^{**})$.  
	
	Next, we demonstrate the other inclusion.  Fix $w\in W_n^{**}$ and, for a contradiction, assume that $\phi(w) \notin \Pi_n(1/24/3)$.  Set $\phi(w) = B_1/\cdots/B_m$.  As $\phi(w)$ contains an occurrence of $1/24/3$ we must have $m\geq 3$.  As $\max(B_s)<\min(B_{s+1})$ for all $s>1$ we see that any occurrence of $1/24/3$ in $\phi(w)$ must involve $B_1$ and two other blocks $B_r$ and $B_s$ with $r<s$.   If the integers that form this pattern are $i<j<k<\ell$ then we have exactly two cases.

	\medskip
	
	{\tt Case 1: } $j,\ell\in B_1$, $i\in B_r$, and $k\in B_s$
	
	\medskip
	
	In this case, we see that as $i\in B_r$, then $w_j=w_\ell=\textbf{c}$.  As the letters between $w_j$ and $w_\ell$ cannot contain an \textbf{a} we conclude that $w_k = \textbf{b}$.  It now follows that $w_j\cdots w_\ell$ must contains a \textbf{cb}.  
	Additionally, the facts that the letters between $w_j$ and $w_\ell$ do not contain an $\textbf{a}$, and that $j<k<\ell$, and that $w_{\min(B_s)} = \textbf{a}$, imply that $\min(B_s)<j$.  Furthermore, as $\min(B_r)\leq i<j$ we see that 
	$$w_1 = w_{\min(B_r)}=w_{\min(B_s)}=\textbf{a},$$
	and $1< \min(B_r)<\min(B_s)<j$.  This contradicts the fact that $w$ satisfies condition 2) in the definition of $W_n^{**}$.  
	
		\medskip
	
	{\tt Case 2: } $i\in B_r, j,\ell\in B_s$ and $k\in B_1$
	
	\medskip
	
	From the definition of $\phi$, we clearly have
	$$w_1 = w_{\min(B_r)}=w_{\min(B_s)}=\textbf{a}.$$
	Moreover, as $i\in B_r$ and $i<k$, then $w_k=\textbf{c}$.  Further as $\ell\neq \min(B_s)$, then $w_\ell = \textbf{b}$.  Lastly, it is clear that the subword $w_{k+1}\cdots w_{\ell-1}$ cannot contain the letter \textbf{a}.  (If it did, then $\ell$ could not be in $B_s$.) This means $w$ contains a \textbf{cb} preceded by at least three \textbf{a}.  Again this contradicts the fact that $w$ satisfies condition 2) in the definition of $W_n^{**}$.  This completes our proof.  
\end{proof}

\begin{theorem}\label{thm:enum1_24_3}
We have 
$$F_{1/24/3}(z) = \sum_{n\geq 1} |\Pi_n(1/24/3)|z^n = \frac{z-4z^2+6z^3-2z^4}{(1-3z+2z^2)^2}.$$
\end{theorem}
\begin{proof}
	By our previous lemma it suffices to enumerate $W_n^{**}$.  To do so we consider three cases depending on the number of \textbf{c}'s.  Clearly, the words in $W_n^{**}$ with no \textbf{c}'s are counted by the expression $\frac{z}{1-2z}$.  Next consider words that contain at least one \textbf{c} and have the additional property that the first \textbf{c} appears after the second \textbf{a} but before the third \textbf{a} (if it exists).  Such words must be of the form
	$$\textbf{a}\textbf{b}\cdots\textbf{b}\textbf{a}\underbrace{\boxed{\textbf{b}'s\mbox{ and }\textbf{c}'s}}_{\mbox{at least 1 \textbf{c}}}\qquad \mbox{or}\qquad \textbf{a}\textbf{b}\cdots\textbf{b}\textbf{a}\underbrace{\boxed{\textbf{b}'s\mbox{ and }\textbf{c}'s}}_{\mbox{at least 1 \textbf{c}}}\ \textbf{a}\ \boxed{\textbf{a}'s\mbox{ and }\textbf{b}'s}.$$
	These two forms are counted by 
	$$\frac{z^2}{1-z}\left(\frac{1}{1-2z}-\frac{1}{1-z}\right)\left(1+\frac{z}{1-2z}\right).$$
	Lastly, we consider the case that our first \textbf{c} is preceded by at least three \textbf{a}'s.  In this case such words must be of the form
	$$\textbf{ab}\cdots\textbf{bab}\cdots\textbf{ba}\ \boxed{\textbf{a}'s\mbox{ and }\textbf{b}'s}\ \textbf{c}\cdots\textbf{c}\qquad\mbox{ or } \qquad \textbf{ab}\cdots\textbf{bab}\cdots\textbf{ba}\ \boxed{\textbf{a}'s\mbox{ and }\textbf{b}'s}\ \textbf{c}\cdots\textbf{c\ a}\ \boxed{\textbf{a}'s\mbox{ and }\textbf{b}'s}.$$
	Together such words are counted by the expression
	$$\frac{z^3}{(1-z)^2}\frac{z}{(1-2z)(1-z)}\left(1+\frac{z}{1-2z}\right).$$
	Adding these three terms together and simplifying yields the desired expression.  
\end{proof}

With the recursive structure of $1/24/3$-avoiding set partitions established, we now conclude the proof the case $|a-b|=2$ in the proof of Theorem~\ref{thm:delta_lessthan_sigma}, by establishing the following lemma.

\begin{lemma}\label{lem:1_24_3_less_singleton}
For $n\geq 5$, we have $|\Pi_n(1/24/3)| < |\Pi_n(1/2/3/4)|$.
\end{lemma}
\begin{proof}
Let $A_n$ be the set of all partitions of $[n]$ with exactly 2 blocks and let $A_n^*$ be the set of all partitions with exactly 3 blocks in which $n$ is a singleton.  (Note $A_n^*$ is a subset of both  $\Pi_n(1/24/3)$ and $\Pi_n(1/2/3/4)$.)  Now let $C_n$ be the set of all partitions in $\Pi_n(1/24/3)$ with the property that removing $n$ gives a partition with at least 3 blocks.  Similarly, let $D_n$ be the set of all partitions in $\Pi_n(1/2/3/4)$ with the property that removing $n$ results in a set partition with exactly 3 blocks.  
	By our above note and the fact that the patterns involved have at least 3 blocks, it follows that
	$$\Pi_n(1/24/3) = \{\beta_n\} \cupdot A_n \cupdot A_n^* \cupdot C_n \qquad\textrm{and}\qquad \Pi_n(1/2/3/4) = \{\beta_n\} \cupdot A_n\cupdot A_n^* \cupdot D_n.$$
	It now suffices to show that $|C_n|<|D_n|$ for $n\geq 5$.  We proceed by induction on $n$.  As $|\Pi_5(1/24/3)| = 39$ and $|\Pi_5(1/2/3/4)| = 41$, we must have $|C_5| < |D_5|$.  From the description of $1/24/3$-avoiding permutations given in the first paragraph of this section, it follows that $|C_{n+1}|\leq 3|C_n|$.  Additionally, it is clear that $|D_{n+1}| = 3|D_n|$.  Together we get 
	$$|C_{n+1}| \leq 3|C_n| < 3|D_n| =  |D_{n+1}|$$
	where the second inequality follows by our inductive hypothesis.  (Note the first inequality may not be an equality as in Example~\ref{ex:1_24_3}.)
\end{proof}

\subsection{The pattern $12/34$}\label{sec:12_34}
\begin{theorem}\label{thm:enum12_34}
We have 
$$|\Pi_n(12/34)| = \sum_{k=0}^{\lfloor n/2\rfloor} k!{n \choose 2k}+\sum_{\ell=3}^{n-2k}{n\choose 2k+\ell}k!(k+1)^2.$$
\end{theorem}
Before proving this result it will be helpful to state and prove a couple of lemmas.  

\begin{lemma}
	Any set partition which avoids $12/34$ has at most one block of size greater than $2$.  
\end{lemma}
\begin{proof}
	For a contradiction assume $\pi\in \Pi_n(12/34)$ contains two blocks $B=\{x_1,\ldots, x_a\}$ and $C=\{y_1,\ldots, y_b\}$ where $a,b\geq 3$.  Without loss of generality we may further assume that $x_2<y_2$.  On the other hand, this implies that the blocks $B$ and $C$ must contain an occurrence of $12/34$ which is impossible.  
\end{proof}

The proof of the next lemma is straightforward.  The details are left to the reader.  
\begin{lemma}\label{lem:blockpositions12_34}
	Let $B=\{z_1<\cdots< z_a\}$, $C=\{x_1<y_1\}$, and $D=\{x_2<y_2\}$ be blocks in $\pi \in \Pi_n(12/34)$ so that $3\leq a$.  Then 
	\begin{enumerate}
\item[a)] $\max(x_1,x_2)<\min(y_1,y_2)$.  
\item[b)] $x_1<z_2$ and $z_{a-1}<y_1$.  
	\end{enumerate}
\end{lemma}

\begin{remark}
It immediately follows from the previous lemma that if $x_1<\cdots<x_k$ are the minimum entries among all the blocks of size 2 and $y_1,\ldots, y_k$ are the maximum entries among all the blocks of size 2 then we must have 
$$x_i < \min(y_1,\ldots, y_k),$$	
for $1\leq i\leq k$.  
\end{remark}

\begin{proof}[Proof of Theorem~\ref{thm:enum12_34}]
To start, let us first count those set partitions in $\Pi_n(12/34)$ whose block sizes do not exceed 2.  To count such set partitions with exactly $k$ blocks of size 2, we first choose a subset $x_1<\cdots<x_k<y_1<\cdots y_k$ of size $2k$ from $[n]$. (The $n-2k$ integers not chosen become singletons.) We then may choose to match up each of the $x_i$'s with exactly one of the $y_i$'s in any of the $k!$ ways.  It follows from the above remark that all set partitions in $\Pi_n(12/34)$ whose block sizes do not exceed 2 are of this form.  A simple argument further shows that any set partition built in this manner must also avoid $12/34$.  Consequently, the number of such set partitions is given by
$$\sum_{k=0}^{\lfloor n/2\rfloor}k!{n \choose 2k}.$$

Now let us consider the set partitions in $\Pi_n(12/34)$ that contain $k$ blocks of size 2 and exactly one block of size $\ell\geq 3$.  To build such a partition, we first choose a subset of size $2k+\ell$ from $[n]$.  Let us denote the members of this subset by 
$$x_1<\cdots< x_{k+1}< z_1 <\cdots < z_{\ell-2} < y_1<\cdots <y_{k+1}.$$
(Again, the $n-2k-\ell$ integers not chosen become singletons in our final set partition.)
Next, choose exactly one element from the $x_i$'s and one element from the $y_i$'s along with all the $z_i$'s to form our block of size $\ell$.  Next, as in the preceding case we are free to match each of the remaining $x_i$'s to each of the remaining $y_i$'s in all $k!$ possible ways to form our $k$ blocks of size 2.  Lemma~\ref{lem:blockpositions12_34} guarantees that any such set partition is built in this manner.  Furthermore, it is straightforward to show that any set partition built in this manner avoids $12/34$.  Consequently, the number of such set partitions is given by 
$$\sum_{k=0}^{\lfloor n/2\rfloor}\sum_{\ell=3}^{n-2k}{n\choose 2k+\ell}k!(k+1)^2.$$
Adding these two terms gives our final result.  
	
\end{proof}

\subsection{The pattern $14/23$}\label{sec:14_23}
To begin, let us concentrate on the subset $\Pi_n^*(14/23)$ consisting of all set partitions in $\Pi_n(14/23)$ that do not contain singletons.  We first show that the set partitions in this subset can be constructed via a simple recursive procedure.  To do this we first need the following definition.
\begin{definition}
	For any $\pi\in\Pi_n^*(14/23)$ we say $k$ is a \emph{cap} provided each $i\geq k$ is the maximum element in its block.  Let $\iota(\pi)$ be the number of caps in $\pi$.  
\end{definition}

Next, we define the sets $\Pi_n^*$ as follows.  First, set $\Pi_2^* = \{\{1,2\}\}$.  For $n\geq 3$, define $\Pi_n^*$ to be the set of all partitions obtained by either of the following two operations.  
\medskip

\noindent
{\tt $n$-insertion}
\medskip

For any $\pi\in \Pi_{n-1}^*$ insert $n$ into the block containing $n-1$.
\medskip

\noindent
{\tt $n$-augmentation}
\medskip

For $\sigma \in \Pi_{n-2}^*$ where $n\geq 4$ doing the following.  Fix $k$ to be either one of the $\iota(\pi)$ caps in $\sigma$ or set $k=n-1$. Then, increment all the values in $\sigma$ which are $\geq k$.  Finally, append the doubleton block $\{k,n\}$.

We note that no partition in $\Pi_n^*$ contains a singleton.  

\begin{lemma}
	We have $\Pi_n^* = \Pi^*_n(14/23)$.  
\end{lemma}
\begin{proof}
	We leave it to the reader to convince themselves that $\Pi_n^* \subseteq \Pi_n^*(14/23)$.  We show the other inclusion by induction on $n$.  First note that $\Pi_2^* = \{\{1,2\}\} = \Pi_2^*(14/23)$.  Now take any $\pi\in \Pi_{n+1}^*(14/23)$ and let $B$ be the block whose cap is $n+1$.  If $|B|\geq 3$, then we claim that $n\in B$ as well.  It then follows (inductively) that $\pi\in \Pi_{n+1}^*$.  To prove this claim, let $B = \{x_1<x_2<\cdots < x_k\}$ so that $k\geq 3$ and $x_k = n+1$.  For a contradiction assume $x_{k-1}<n$.  So there exists a block $C$, distinct from $B$, whose cap is $n$.  Set $C = \{y_1<\cdots<y_\ell\}$ so that $y_\ell = n$ and $\ell\geq 2$, as $\pi$ does not contain singletons.  We must have 
	$$y_{\ell-1} < x_{k-2}<x_{k-1} < n = y_\ell\quad\mbox{or}\qquad x_{k-2}<y_{\ell-1}<y_\ell=n<n+1 = x_k.$$
	But either choice results in an occurrence of the forbidden pattern $14/23$.  Hence $n\in B$ as claimed.  
	
	The other possibility is for $B=\{k,n+1\}$.  As $\pi$ avoids $14/23$ it follows that every $i$ strictly between $k$ and $n+1$ must be a cap.  Consequently, $\pi$ was constructed from some set partition in $\Pi_n^*(14/23)= \Pi_n^*$ via $n+1$-augmentation.  
\end{proof}
With this lemma established, we are now ready to enumerate this pattern.  
\begin{theorem}
	We have 
	$$F_{14/23}(z) =\sum_{n\geq 0} |\Pi_n(14/23)|\ z^n =  G\left(\frac{z}{1-z}\right)\frac{1}{1-z} + \frac{1}{1-z},$$
	where $$G(z) = \frac{z-2z^2(1+z)-z\sqrt{1-4z^2}}{-2+2z(1+z)^2}.$$
\end{theorem}
\begin{proof}
	In light of the previous lemma we know that 
	$$G(z) = \sum_{n\geq 2} |\Pi^*_n(14/23)|\ z^n = \sum_{n\geq 2} |\Pi^*_n|\  z^n.$$
	It is straightforward to see that an arbitrary set partition in $\Pi_n(14/23)$ is obtained by inserting $k$ singleton blocks into some set partition in $\Pi_{n-k}(14/23)$.  In terms of generating functions this corresponds to 
	$$F_{14/23}(z) = G\left(\frac{z}{1-z}\right)\frac{1}{1-z} + \frac{1}{1-z}.$$
	It only remains to prove that $G$ is given by the desired generating function.  To do this, first define
	$$H(z,t) = \sum_{n\geq 2}\sum_{\pi \in \Pi_n^*}z^n t^{\iota(\pi)}.$$
	The recursive description of $\Pi_n^*$ translates into the functional equation
	$$H(z,t) = z^2t +ztH(z,1) + \frac{z^2t}{1-t}\left(H(z,1) -tH(z,t)\right),$$
	where the second term corresponds to $n$-insertion and the third term corresponds to $n$-augmentation.  Solving for $G(z) = H(z,1)$ using the kernel method results in the desired expression.  
\end{proof}

\subsection{The pattern $1/234$}\label{sec:1_234}
\begin{theorem}
	We have 
	$$|\Pi_n(1/234)| = \sum_{\ell=1}^{n} M(n-\ell) + \sum_{\ell=1}^{n-2} (n-\ell-1)M(n-\ell-1)+\sum_{\ell=1}^{n-3} {n-\ell-1 \choose 2}M(n-\ell-2),$$	
	where 
	$$M(n) = \sum_{k=0}^{\lfloor n/2 \rfloor} {n\choose 2k} (2k)!!\ .$$
\end{theorem}
\begin{proof}
Let $\pi = B_1/\cdots /B_m \in \Pi_n(1/234)$.  First observe that $|B_i|\leq 2$ for all $i\geq 2$ since $1\in B_1$ and $\pi$ avoids $1/234$.  This means that the set partition $B_2/\cdots /B_m$ is a matching with fixed points of the set $[n]\setminus B_1$ which has size $n_0 = n-|B_1|$.  It is a well known result that such objects are counted by 
$$M(n_0) = \sum_{k=0}^{\lfloor n_0/2 \rfloor} {n_0\choose 2k} (2k)!!\ .$$
Next, observe that $B_1$ is not free to be any subset of $[n]$.  In fact $B_1$ must be of either the form
\begin{enumerate}
	\item[1)] $\{1,\ldots, \ell\}$, or
	\item[2)] $\{1,\ldots, \ell, k\}$ for some $k>\ell+1$, or
	\item[3)] $\{1,\ldots, \ell, k,m\}$ for some $m>k>\ell+1$,
\end{enumerate}
as any other possibility would result in an occurrence of $1/234$.  

Combining these two observations we see that the first term in our formula counts the set partitions in $\Pi_n(1/234)$ whose first block is of the form in 1).  Likewise, the second and third terms count those set partitions whose first block is of the form in 2) and 3) respectively.  
\end{proof}

\subsection{The pattern $134/2$}\label{sec:134_2}
To enumerate this pattern we require the well studied notion of weak integer compositions.  In particular, we denote by $C_{n,k}$ the set of all weak integer compositions of $n$ with $k$ parts. Additionally we denote by $M_{k,f}$ the set of all set partitions of $[2k+f]$ with $f$ singletons and $k$ doubletons.  (Observe these are just matchings with $f$ fixed points.)  Lastly, we define the refined set $\Pi_{n,k,f}(134/2)$ to be the set of all set partitions in $\Pi_{n}(134/2)$ with exactly $k+f$ blocks where exactly $f$ of them are singletons.  

\begin{lemma}\label{lem:bijection134_2}
	Provided, $k\geq 1$ and $2k+f\leq n$, there exists an explicit bijection 
	$$\phi:\Pi_{n,k,f}(134/2)\to C_{n-f-2k,k} \times M_{k,f} .$$
\end{lemma}

Deferring the proof of this lemma to the end of this section, we continue with our enumeration.

As it is well known that $|C_{n-f-2k,k}| = {n-f-k-1\choose k-1}$
and $|M_{k,f}| = {2k+f\choose f}(2k)!!$
it now follows from Lemma~\ref{lem:bijection134_2}, that
\begin{align*}
	|\Pi_n(134/2)| &= 1+ \sum_{k=1}^{\lfloor n/2 \rfloor}\sum_{f=0}^{n-2k} |\Pi_{n,k,f}(134/2)|\\
	&= 1+ \sum_{k=1}^{\lfloor n/2 \rfloor}\sum_{f=0}^{n-2k} {n-f-k-1\choose k-1}{2k+f\choose f}(2k)!!. \\
\end{align*}

We record this result in our last theorem. 

\begin{theorem}
	We have the following formula
	$$|\Pi_n(134/2)| = 1+ \sum_{k=1}^{\lfloor n/2 \rfloor}\sum_{f=0}^{n-2k} {n-f-k-1\choose k-1}{2k+f\choose f}(2k)!!.$$
\end{theorem}

We now turn our attention to the proof of Lemma~\ref{lem:bijection134_2}.  We begin with a simple characterization of $134/2$-avoiding set partitions.  Its straightforward proof is omitted.

\begin{lemma}\label{lem:characterization134_2}
	For any set partition $\pi = B_1/\cdots /B_m$, we have that $\pi$ is $134/2$-avoiding if and only if any non-singleton block is of the form
	$$\{a, a+1,\ldots, a+\ell, b\},$$
	where $\ell\geq 0$ and $b\geq a+\ell+1$. 
\end{lemma}

\begin{proof}[Proof of Lemma~\ref{lem:bijection134_2}]
Consider a set partition $\pi=B_1/\ldots/B_m\in \Pi_n(134/2)$ so that $f$ of the blocks are singletons and the remaining $k= m-f$ blocks  $B_{i_1},\ldots, B_{i_k}$ are not singletons.  Now let
$$\lambda= \left(|B_{i_1}|-2,|B_{i_2}|-2,\ldots, |B_{i_k}|-2\right)$$
be the resulting weak composition of $n-f-2k$ with $k$ parts.  (As each of the blocks $B_{i_j}$ are of the form in Lemma~\ref{lem:characterization134_2}, the parts of our composition are just their corresponding values for $\ell$ in this decomposition.) Furthermore, by throwing out all but the min and max of each block, and applying standardization map, we obtain a  set partition $\sigma$  of $[f+2k]$ with exactly $f$ singletons and $k$ doubletons.  Finally, we define $\phi(\pi)=(\lambda,\sigma)$.  (We illustrate this construction in Example~\ref{ex:phi}.) As this map is easily seen to be bijective, the proof is complete.  
\end{proof}

\begin{example}\label{ex:phi}
Consider the set partition  $127/3/4568/9/10\ 11\ 12$ where $f=2$ and $k = 3$.  Then
	$$\phi(127/3/4568/9/10\ 11\ 12) = (\lambda, \sigma)$$
	where $\lambda = (1,2,1) \in C_{4,3}$ and $\sigma = 14/2/35/6/78\in M_{3,2}$.  
\end{example}

\acknowledgements
\label{sec:ack}
The authors are grateful to Bruce Sagan for recommending this area of research.

\bibliographystyle{abbrvnat}
\bibliography{mybib}
\label{sec:biblio}

\end{document}